\documentclass{amsart}
\usepackage{amsmath}
\usepackage{amsfonts}
\usepackage{amssymb,enumerate}
\usepackage{amsthm}
\usepackage[all]{xy}
\usepackage{hyperref}

\newcommand{\id}{\operatorname{id}}

\newcommand{\Gid}{\operatorname{Gid}}

\newcommand{\pd}{\operatorname{pd}}

\newcommand{\Hom}{\operatorname{Hom}}

\newcommand{\Ext}{\operatorname{Ext}}
\newcommand{\Ker}{\operatorname{Ker}}

\newcommand{\Ass}{\operatorname{Ass}}

\newcommand{\Spec}{\operatorname{Spec}}
\newcommand{\Supp}{\operatorname{Supp}}
\newcommand{\supp}{\operatorname{supp}}
\newcommand{\depth}{\operatorname{depth}}
\newcommand{\width}{\operatorname{width}}

\renewcommand{\dim}{\operatorname{dim}}

\newcommand{\E}{\operatorname{E}}
\newcommand{\FID}{\operatorname{FID}}
\newcommand{\FGID}{\operatorname{FGID}}
\renewcommand{\H}{\operatorname{H}}
\newcommand{\uhom}{{\mathbf R}\Hom}
\newcommand{\utp}{\otimes^{\mathbf L}}

\renewcommand{\H}{\mbox{H}}

\newcommand{\fa}{\frak{a}}
\newcommand{\fn}{\frak{n}}
\newcommand{\fm}{\frak{m}}
\newcommand{\fp}{\frak{p}}

\newcommand{\fq}{\frak{q}}
\newtheorem{thm}{Theorem}[section]
\newtheorem{cor}[thm]{Corollary}
\newtheorem{lem}[thm]{Lemma}
\newtheorem{prop}[thm]{Proposition}

\newtheorem{defn}[thm]{Definition}

\newtheorem{rem}[thm]{Remark}

\begin{document}

\bibliographystyle{amsplain}

\date{}
\author{Afsaneh Esmaeelnezhad and Parviz Sahandi}

\address{Department of Mathematics, University of Kharazmi, Tehran, Iran
}

\email{esmaeilnejad@gmail.com }

\address{
Department of Mathematics, University of Tabriz, Tabriz, Iran, and
School of Mathematics, Institute for Research in Fundamental
Sciences (IPM), P.O. Box: 19395-5746, Tehran Iran. }

\email{sahandi@tabrizu.ac.ir, sahandi@ipm.ir}



\thanks{P. Sahandi was supported in part by a grant from
IPM (No. 91130030).}

\title{On graded Gorenstein injective dimension}

\begin{abstract}
There are nice relations between graded homological dimensions and
ordinary homological dimensions. We study the Gorenstein injective
dimension of a complex of graded modules denoted by $^*\Gid$, and
derive its properties. In particular we prove the Chouinard's like
formula for $^*\Gid$, and compare it with the usual Gorenstein
injective dimension.
\end{abstract}

\maketitle

\section{Introduction}


Let $R$ be a Noetherian $\mathbb{Z}$-graded ring. In \cite{F} and
\cite{FF}, Fossum and Fossum-Foxby have studied the graded
homological dimension of graded modules and compare them with
classical homological dimensions. They shown that for a graded
$R$-module $M$, one has $\,^*\id_{R}M\leq
\id_{R}M\leq\,^{*}\id_{R}M+1$ where $\id_{R}M$ (resp.
$\,^*\id_{R}M$) denotes for the injective dimension of $M$ in the
category of $R$-modules (resp. category of graded $R$-modules). It
is natural to ask how these inequalities hold for the Gorenstein
injective dimension $\Gid_RM$. In this paper we give an answer this
question. Section 2 of this paper is devoted to review some
hyper-homological algebra for the derived category of the graded
ring $R$. In Section 3 we define the $^*$injective dimension of
complexes of graded modules and homogeneous homomorphisms, and
derived its properties. In the final section we define the graded
Gorenstein injective dimension of complexes of graded modules and
homogeneous homomorphisms denoted by $\,^*\Gid$. Among other results
we show that if the graded ring $R$ admits a $\!^*$dualizing
complex, or is a non-negatively graded ring, then $\,^*\Gid_{R}X\leq
\Gid_{R}X\leq \,^*\Gid_{R}X+1$, where $\Gid_{R}X$ is the Gorenstein
injective dimension of $X$ over $R$ (see Corollary \ref{ingid}).
Also in this case we prove a Chouinard's like formula for
$\,^*\Gid_{R}X$ (see Theorem \ref{C}). Our source of graded rings
and modules are \cite{BH} and \cite{FF}.

Throughout this paper $R$ is a commutative Noetherian
$\mathbb{Z}$-graded ring.

\section{Derived category of complexes of graded modules}


We use the notation from the appendix of \cite{C}. Let $X$ be a
complex of $R$-modules and $R$-homomorphisms. The \emph{supremum}
and the \emph{infimum} of a complex $X$, denoted by $\sup(X)$ and
$\inf(X)$ are defined by the supremum and infimum of
$\{i\in\mathbb{Z}|\H_i(X)\neq0\}$. If $m$ is an integer and $X$ is a
complex, then $\Sigma^mX$ denotes the complex $X$ \emph{shifted} $m$
degrees to the left; it is given by
$$
(\Sigma^mX)_{\ell}=X_{\ell-m}\text{ and
}\partial^{\Sigma^mX}_{\ell}=(-1)^m\partial^X_{\ell-m}
$$
for $\ell\in \mathbb{Z}$.

The symbol ${\mathcal D}(R)$ denotes the \emph{derived category} of
$R$-complexes. The full subcategories ${\mathcal D}_{\sqsubset}
(R)$, ${\mathcal D}_{\sqsupset} (R)$, ${\mathcal D}_{\square} (R)$
and ${\mathcal D}_{0} (R)$ of ${\mathcal D}(R)$ consist of
$R$-complexes $X$ while $\H_{\ell}(X)=0$, for respectively $\ell\gg
0$, $\ell\ll 0$, $|\ell|\gg 0$ and $\ell\neq 0$. Homology
isomorphisms are marked by the sign $\simeq$. The right derived
functor of the homomorphism functor of $R$-complexes and the left
derived functor of the tensor product of $R$-complexes are denoted
by $\uhom_{R}(-,-)$ and $-\utp_{R}-$, respectively.

Let $M=\oplus_{n\in \mathbb{Z}}M_n$ and $N=\oplus_{n\in
\mathbb{Z}}N_n$ be two graded $R$-modules. The $^*\Hom$ functor is
defined by $^*\Hom_R(M,N)=\bigoplus_{i\in \mathbb{Z}}\Hom_i(M,N)$,
such that $\Hom_i(M,N)$ is a $\mathbb{Z}$-submodule of $\Hom_R(M,N)$
consisting of all $\varphi:M\to N$ such that $\varphi(M_n)\subseteq
N_{n+i}$ for all $n\in \mathbb{Z}$. In general
$^*\Hom_R(M,N)\neq\Hom_R(M,N)$ but equality holds if M is finitely
generated, see \cite[Exercise 1.5.19]{BH}. Also the tensor product
$M\otimes_RN$ of $M$ and $N$ is a graded module with
$(M\otimes_RN)_n$ is generated (as a $\mathbb{Z}$-module) by
elements $m\otimes n$ with $m\in M_i$ and $n\in N_j$ where $i+j=n$.

Let $\{M_{\alpha}\}_{\alpha\in I}$ be a family of graded
$R$-modules. Then $\bigoplus_{\alpha}M_{\alpha}$ becomes a graded
$R$-module with
$(\bigoplus_{\alpha}M_{\alpha})_n=\bigoplus_{\alpha}(M_{\alpha})_n$
for all $n\in \mathbb{Z}$, see \cite[Page 289]{FF}. Also recall that
the direct product exists in the category of graded modules. Then
the direct product is denoted by $^*\prod_{\alpha}M_{\alpha}$ and
$(^*\prod_{\alpha}M_{\alpha})_n=\prod_{\alpha}(M_{\alpha})_n$ for
all $n\in \mathbb{Z}$, see \cite[Page 289]{FF}. In this case there
are the following bijections \cite[Page 289]{FF}
$$
^*\Hom_R(\bigoplus_{\alpha}M_{\alpha},-)\stackrel{\cong}{\longrightarrow}\!^*\prod_{\alpha}\!^*\Hom_R(M_{\alpha},-),
$$
$$
^*\Hom_R(-,\!^*\prod_{\alpha}M_{\alpha})\stackrel{\cong}{\longrightarrow}\!^*\prod_{\alpha}\!^*\Hom_R(-,M_{\alpha}).
$$
Likewise direct and inverse limits are exists in the category of
graded modules with
$$
(\!^*\lim_{\longrightarrow}M_{\alpha})_n=\lim_{\longrightarrow}(M_{\alpha})_n,
$$
$$
(\!^*\lim_{\longleftarrow}M_{\alpha})_n=\lim_{\longleftarrow}(M_{\alpha})_n,
$$
see \cite[Page 289]{FF}. Let $(R,\fm)$ be a $^*$local (Noetherian)
ring $R$, that is, a graded ring with a unique homogeneous maximal
ideal $\fm$. The $\fm$-$^*$adic completion of $R$ is
$$
\!^*\widehat{R}=\!^*\lim_{\longleftarrow}R/\fm^n,
$$
which is a Noetherian graded ring by \cite[Corollary VIII.2]{F}. It
is known that the $\fm$-$^*$adic completion $\,^*\widehat{R}$ is a
flat $R$-module \cite[Corollary 3.3]{FF}, and that if
$E:=\!^*\E_R(R/\fm)$ is the $^*$ injective envelope of $R/\fm$ over
$R$, then $\,^*\Hom_R(E,E)\cong \!^*\widehat{R}$ \cite[Theorem
VIII.3]{F}.

The symbol $^*{\mathcal C}(R)$ denotes the category of complexes of
graded $R$-modules and homogeneous differentials. Note that the
category of graded modules is an abelian category, hence
$^*{\mathcal C}(R)$ has a derived category, (see \cite{Ha}), which
will be denoted by $\,^*{\mathcal D}(R)$. Analogously we have
$\,^*{\mathcal C}_{\sqsubset} (R)$, $\,^*{\mathcal C}_{\sqsupset}
(R)$, $\,^*{\mathcal C}_{\square} (R)$ and $\,^*{\mathcal C}_{0}
(R)$ (resp. $\,^*{\mathcal D}_{\sqsubset} (R)$, $\,^*{\mathcal
D}_{\sqsupset} (R)$, $\,^*{\mathcal D}_{\square} (R)$ and
$\,^*{\mathcal D}_{0} (R)$) which are the full subcategories of
$\,^*{\mathcal C}(R)$ (resp. $\,^*{\mathcal D}(R)$). If we use the
notation $X\in ^*\mathcal{C}_{(\sharp)}(R)$, we mean $\H(X)\in
^*\mathcal{C}_{\sharp}(R)$.

For $R$-complexes $X$ and $Y$ of graded modules, with homogeneous
differentials $\partial^X$ and $\partial^Y$, we define the
\emph{homomorphism complex} $^*\Hom_R(X,Y)$ as follows:
$$
^*\Hom_R(X,Y)_{\ell}=\!^*\prod_{p\in
\mathbb{Z}}\!^*\Hom_R(X_p,Y_{p+\ell})
$$
and when $\psi=(\psi_p)_{p\in \mathbb{Z}}$ belongs to
$^*\Hom_R(X,Y)_{\ell}$ the family
$\partial_{\ell}^{^*\Hom_R(X,Y)}(\psi)$ in $^*\Hom_R(X,Y)_{\ell-1}$
has $p$-th component
$$
\partial_{\ell}^{^*\Hom_R(X,Y)}(\psi)_p=\partial^Y_{p+\ell}\psi_p-(-1)^{\ell}\psi_{p-1}\partial^X_p.
$$

When $X\in\,^*{\mathcal C}_{\sqsupset}^f(R)$ and $Y\in\,^*{\mathcal
C}_{\sqsubset}(R)$ all the products $\!^*\prod_{p\in
\mathbb{Z}}\!^*\Hom_R(X_p,Y_{p+\ell})$ are finite. Thus using
\cite[Exercise 1.5.19]{BH}, we have
$$
\!^*\prod_{p\in
\mathbb{Z}}\!^*\Hom_R(X_p,Y_{p+\ell})=\bigoplus_{p\in
\mathbb{Z}}\Hom_R(X_p,Y_{p+\ell}),
$$
for every $\ell\in \mathbb{Z}$. Therefore
$^*\Hom_R(X,Y)=\Hom_R(X,Y)$.

We also define the \emph{tensor product complex} $X\otimes_RY$ as
follows:
$$
(X\otimes_RY)_{\ell}=\bigoplus_{p\in
\mathbb{Z}}(X_p\otimes_RY_{\ell-p})
$$
and the $\ell$-th differential $\partial_{\ell}^{X\otimes_RY}$ is
given on a generator $x_p\otimes y_{\ell-p}$ in
$(X\otimes_RY)_{\ell}$, where $x_p$ and $y_{\ell-p}$ are homogeneous
elements, by
$$
\partial_{\ell}^{X\otimes_RY}(x_p\otimes y_{\ell-p})=\partial^X_p(x_p)\otimes y_{\ell-p}+(-1)^px_p\otimes\partial^Y_{\ell-p}(y_{\ell-p}).
$$
If $X$ and $Y$ are $R$-complexes of graded modules, then
$\,^*\Hom_R(X,-)$, $\,^*\Hom_R(-,Y)$, and $X\otimes_R-$ are functors
on $\,^*{\mathcal C}(R)$.

Note that any object of $\,^*{\mathcal C}_{\sqsubset} (R)$ has an
$^*$injective resolution by \cite[Page 47]{Ha}, and any object of
$\,^*{\mathcal C}_{\sqsupset} (R)$ has an $^*$projective resolution
by \cite[Page 48]{Ha}. The right derived functor of the $\!^*\Hom$
functor in the category of graded complexes is denoted by
$\mathbf{R}\,^*\Hom_{R}(-,-)$ and set
$\,^*\Ext^i_R(-,-)=\H_{-i}(\mathbf{R}\,^*\Hom_{R}(-,-))$. It is
easily seen that if $R$ is a Noetherian $\mathbb{Z}$-graded ring and
$X$ a homologically finite complex of graded modules and
$Y\in\!^*\mathcal{C}(R)$ then
$\mathbf{R}\,^*\Hom_{R}(X,Y)=\uhom_R(X,Y)$. Also the left derived
functor of $-\otimes_R-$ in the category of graded complexes is
denoted by $-\otimes_R^{\mathbf{L}^*}-$. Since $^*$projective graded
$R$-modules coincide with projective $R$-modules by
\cite[Proposition 3.1]{FF} we easily see that
$-\otimes_R^{\mathbf{L}^*}-$ coincides with the ordinary left
derived functor of $-\otimes_R-$ in the category of complexes. So we
use $-\utp_R-$ instead of $-\otimes_R^{\mathbf{L}^*}-$.

For the homomorphism and the tensor product functors we have the
following useful proposition, see \cite[A.2.8, A.2.10, and
A.2.11]{C} for the ungraded case. The proof is the same as the
ungraded case so we omit it.

\begin{prop} Let $S$ be a graded ring which is an $R$-algebra.

{\bf $^*$Adjointness.} Let $Z, Y\in\,^*{\mathcal C}(S)$ and
$X\in\,^*{\mathcal C}(R)$. Then $Z\otimes_SY\in\,^*{\mathcal C}(S)$
and $\!^*\Hom_R(Y,X)\in\,^*{\mathcal C}(S)$, and there is an
isomorphism of $S$-complexes
$$
\rho_{ZYX}:^*\Hom_R(Z\otimes_SY,X)\stackrel{\cong}{\longrightarrow}\!^*\Hom_S(Z,\!^*\Hom_R(Y,X)),
$$
which is natural in $Z, Y$ and $X$.

{\bf $^*$Tensor-Evaluation.} Let $Z, Y\in\,^*{\mathcal C}(S)$ and
$X\in\,^*{\mathcal C}(R)$. Then $^*\Hom_S(Z,Y)\in\,^*{\mathcal
C}(R)$ and $Y\otimes_RX\in\,^*{\mathcal C}(S)$, and there is a
natural morphism of $S$-complexes
$$
\omega_{ZYX}:^*\Hom_S(Z,Y)\otimes_RX\longrightarrow\!^*\Hom_S(Z,Y\otimes_RX).
$$
The morphism is invertible under each of the next two extra
conditions
\begin{itemize}
\item[(i)] $Z\in\!^*\mathcal{C}_{\square}^{fp}(S)$,
$Y\in\!^*\mathcal{C}_{\sqsupset}(S)$, and
$X\in\!^*\mathcal{C}_{\sqsupset}(R)$; or
\item[(ii)] $Z\in\!^*\mathcal{C}_{\square}^{fp}(S)$,
$Y\in\!^*\mathcal{C}_{\sqsubset}(S)$, and
$X\in\!^*\mathcal{C}_{\square}(R)$.
\end{itemize}

{\bf $^*$Hom-Evaluation.} Let $Z, Y\in\,^*{\mathcal C}(S)$ and
$X\in\,^*{\mathcal C}(R)$. Then $^*\Hom_S(Z,Y)\in\,^*{\mathcal
C}(R)$ and $\,^*\Hom_R(Y,X)\in\,^*{\mathcal C}(S)$, and there is a
natural morphism of $S$-complexes
$$
\theta_{ZYX}:Z\otimes_S\!^*\Hom_R(Y,X)\longrightarrow\!^*\Hom_R(\!^*\Hom_S(Z,Y),X).
$$
The morphism is invertible under each of the next two extra
conditions
\begin{itemize}
\item[(i)] $Z\in\!^*\mathcal{C}_{\square}^{fp}(S)$,
$Y\in\!^*\mathcal{C}_{\sqsupset}(S)$, and
$X\in\!^*\mathcal{C}_{\sqsubset}(R)$; or
\item[(ii)] $Z\in\!^*\mathcal{C}_{\sqsupset}^{fp}(S)$,
$Y\in\!^*\mathcal{C}_{\sqsubset}(S)$, and
$X\in\!^*\mathcal{C}_{\square}(R)$.
\end{itemize}
By $Z\in\!^*\mathcal{C}^{fp}(S)$ we mean that $Z$ consists of
finitely generated projective $S$-modules.
\end{prop}




We recall the definition of the \emph{depth} and \emph{width} of
complexes. Let $\fa$ be an ideal in a ring $R$ and $X$ a complex of
graded $R$-modules. The $\fa$-$\depth$ and $\fa$-$\width$ of $X$
over $R$ are defined respectively by
\begin{align*}
\depth(\fa,X):=&-\sup\uhom_R(R/{\fa},X),\\
\width(\fa,X):=&\inf(R/{\fa}\utp_RX).
\end{align*}
If $(R,\fm)$ is a local ring then set $\depth_RX:=\depth(\fm,X)$ and
$\width_RX:=\width(\fm,X)$.

Now let $(R,\fm)$ be a $\,^*$local graded ring and $X$ be a complex
of graded $R$-modules. By \cite[Proposition 1.5.15(c)]{BH},
$-\otimes_RR_{\fm}$ is a faithfully exact functor on the category of
graded $R$-modules. Then we have
\begin{align*}
\width(\fm,X)=& \inf\{i|\H_i(R/\fm\utp_RX)\neq0\} \\[1ex]
= & \inf\{i|\H_i(R/\fm\utp_RX)\otimes_RR_{\fm}\neq0\}\\[1ex]
= & \inf\{i|\H_i(R_{\fm}/\fm R_{\fm}\utp_{R_{\fm}}X_{\fm})\neq0\}\\[1ex]
= & \width(\fm R_{\fm},X_{\fm})=\width _{R_{\fm}}X_{\fm}.
\end{align*}
Likewise we have $\depth(\fm,X)=\depth _{R_{\fm}}X_{\fm}$.

\begin{prop}\label{matlis} Let $(R,\fm,k)$ be a $^*$local ring and
$X\in\mathcal{C}_{\sqsupset}(R)$ and set $E:=\!^*\E_R(k)$. Then
$$\width(\fm,X)=\depth(\fm,\mathbf{R}\!^*\Hom_R(X,E)).$$
\end{prop}

\begin{proof} We have the following computations:
\begin{align*}
\width(\fm,X)=& \inf(k\utp_RX)=-\sup\mathbf{R}\!^*\Hom_R(k\utp_RX,E) \\[1ex]
= & -\sup\mathbf{R}\!^*\Hom_R(k,\mathbf{R}\!^*\Hom_R(X,E)) \\[1ex]
= & \depth(\fm,\mathbf{R}\!^*\Hom_R(X,E)).
\end{align*}
The second equality hods since $E$ is faithfully $^*$injective and
the third one uses the $^*$adjointness isomorphism.
\end{proof}

The following lemma is the graded version of one of Foxby's
accounting principles \cite[Lemma A.7.9]{C}.

\begin{lem}\label{ac} Let $(R,\fm,k)$ be a $^*$local ring. Then
\begin{itemize}
\item[(a)] If $X\in \,^*\mathcal{C}(k)$, then there is a
quasiisomorphism $\H(X)\rightarrow X$.
\item[(b)] If $X\,^*\mathcal{C}_{(\sqsubset)}(R)$ and $W\in \,^*\mathcal{C}(k)$ then
$$\inf\mathbf{R}\!^*\Hom_R(W,X)=\inf\mathbf{R}\!^*\Hom_R(k,X)-\sup W.$$
\end{itemize}
\end{lem}

\begin{proof} Part $(a)$ is easy since very graded $k$-module is free by
\cite[Exercise 1.5.20]{BH}. For part $(b)$ note that $W=k\utp_kW$.
Then using the $^*$adjointness isomorphism we have
$$
\mathbf{R}\!^*\Hom_R(W,X)=\mathbf{R}\!^*\Hom_k(W,\mathbf{R}\!^*\Hom_R(k,X))=\!^*\Hom_k(W,\mathbf{R}\!^*\Hom_R(k,X)).
$$
Thus using part $(a)$ we have
\begin{align*}
\inf\mathbf{R}\!^*\Hom_R(W,X)=& \inf\{\ell|\mathbf{R}\!^*\Hom_R(W,X)_{\ell}\neq0\} \\[1ex]
= & \inf\{\ell|\,^*\prod_{i+j=\ell}\!^*\Hom_k(W_{-j},\mathbf{R}\!^*\Hom_R(k,X)_i)\neq0\} \\[1ex]
= & \inf\mathbf{R}\!^*\Hom_R(k,X)-\sup W.
\end{align*}
This completes the proof.
\end{proof}

The following equality is the graded version of \cite[Lemma 2.6]{Y}.

\begin{prop}\label{Y} Let $(R,\fm,k)$ be a $^*$local ring $X, Y\in \,^*\mathcal{C}_{\square}(R)$, then
$$\width(\fm,\mathbf{R}\!^*\Hom_R(X,Y))=\depth(\fm,X)+\width(\fm,Y)-\depth(\fm,R).$$
\end{prop}

\begin{proof} Use Lemma \ref{ac} and the same argument as
\cite[Lemma 2.6]{Y}.
\end{proof}

Let $a\in R$ be homogeneous and set $\alpha=\deg(a)$. Then the
complex $0\to R(-\alpha)\stackrel{a}{\to}R\to0$ concentrated in
degrees 1 and 0 is called the \emph{Koszul complex} of $a$, and
denoted by $K(a)$, where $R(-\alpha)$ denotes the graded $R$-module
with grading given by $R(-\alpha)_n=R_{n-\alpha}$. Note that
$K(a)\in^*\mathcal{C}(R)$. Now let $\fa$ be a homogeneous ideal of
$R$ and $a_1,\cdots,a_n$ be a set of generators of $\fa$ by
homogeneous elements. The Koszul complex of $\fa$, denoted by
$K:=K(\fa)$, and define as $K(a_1)\otimes_R\cdots\otimes_RK(a_n)$.
It is shown in \cite{FI} that $\width(\fa,X)=\inf(K\otimes_RX)$.

Let $\alpha:X\to Y$ be a homogeneous morphism between $R$-complexes
of graded $R$-modules. The mapping cone of $\alpha$ is a complex of
graded $R$-modules given by
$$
\mathcal{M}(\alpha)_{\ell}:Y_{\ell}\oplus X_{\ell-1}
$$
and
$$
\delta_\ell^{\mathcal{M}(\alpha)}(y_{\ell},
x_{\ell-1})=(\delta_\ell^{Y}(y_{\ell})+\alpha_{\ell-1}(x_{\ell-1}),-\delta_{\ell-1}^{X}(x_{\ell-1})).
$$
Note that $\delta_\ell^{\mathcal{M}(\alpha)}$ is a homogeneous
differentiation from $Y_{\ell}\oplus X_{\ell-1}$ to
$Y_{\ell-1}\oplus X_{\ell-2}$. It is easy to see that the morphism
$\alpha:X\to Y$ between complexes of graded $R$-modules is
quasi-isomorphism if and only if the mapping cone of $\alpha$ is
homologically trivial (see \cite[Lemma A.1.19]{C}). Also for the
covariant and contravariant  functors $^*\Hom_R(V,-)$ and
$^*\Hom_R(-,W)$ we have the following:
$$
\mathcal{M}(\,^*\Hom_R(V,\alpha))=\,^*\Hom_R(V,\mathcal{M}(\alpha)),
$$
$$
\mathcal{M}(\,^*\Hom_R(\alpha,W))=\Sigma^{1}\,^*\Hom_R(\mathcal{M}(\alpha),W).
$$
See \cite[A2.1.2 and A.2.1.4]{C}.

The ungraded version of the following result contained in
\cite[Proposition 2.7]{CFH}.

\begin{prop}\label{quasi-iso}
Let $\mathfrak{B}$ be a class of graded $R$-modules, and
$\alpha:X\to Y$ be a quasiisomorphism between complexes of graded
$R$-modules such that
$$^*\Hom_R(\alpha,V):\,^*\Hom_R(Y,V)\stackrel{\simeq}\longrightarrow\,^*\Hom_R(X,V)$$
is quasiisomorphism for every module $V\in \mathfrak{B}$. Let
$\widetilde{V}\in\, ^*\mathcal{C}(R)$ be a complex consisting of
modules from $\mathfrak{B}$. Then the induced morphism,
$$
^*\Hom_R(\alpha,\widetilde{V}):\,^*\Hom_R(Y,\widetilde{V})\longrightarrow\,^*\Hom_R(X,\widetilde{V}),
$$
is a quasiisomorphism, provided that either
\begin{itemize}
\item[(a)] $\widetilde{V}\in \,^*\mathcal{C}_{\sqsubset}(R)$, or
\item[(b)] $X,Y\in \,^*\mathcal{C}_{\sqsubset}(R)$.
\end{itemize}
\end{prop}

\begin{proof}
In either cases it is enough to show that
$^*\Hom_R(\mathcal{M}(\alpha),\widetilde{V})$ is homologically
trivial. On the other hand by graded version of \cite[Lemma
2.5]{CFH} the result holds if
$^*\Hom_R(\mathcal{M}(\alpha),\widetilde{V}_\ell)$ is homologically
trivial for each $\ell \in \mathbb{Z}$ which this is our assumption.
\end{proof}

\section{$^*$injective dimension}

The injective dimension of a complex $X$ is defined and studied in
\cite{AF}, denoted by $\id_RX$. A graded module $J$ is called
$^*$injective if it is an injective object in the category of graded
modules. The injective dimension of a graded module $M$ in the
category of graded modules, is denoted by $^*\id_RM$ (cf. \cite{FF,
Ha, BH}). The $^*$injective dimension of a complex of graded modules
$X$ is studied in \cite[Page 83]{Ha}. Let $n\in \mathbb{Z}$. A
homologically left bounded complex of graded modules $X$, is said to
have $^*$injective dimension at most $n$, denoted by $^*\id_RX\leq
n$, if there exists an $^*$injective resolution $X\to I$, such that
$I_i=0$ for $i<-n$. If $^*\id_RX\leq n$ holds, but $^*\id_RX\leq
n-1$ does not, we write $^*\id_RX=n$. If $^*\id_RX\leq n$ for all
$n\in \mathbb{Z}$ we write $^*\id_RX=-\infty$. If $^*\id_RX\leq n$
for no $n\in \mathbb{Z}$ we write $^*\id_RX=\infty$. The following
theorem inspired by \cite[Theorem 2.4.I and Corollary 2.5.I]{AF}.

\begin{thm}\label{iid}
For $X\in$ $^{*}\mathcal{D}_{\sqsubset}(R)$ and $n\in \mathbb{Z}$
the following are equivalent:
\begin{itemize}
\item[(1)]$^{*}\id_RX\leq n.$
\item[(2)] $n\geq-\sup U-\inf(\mathbf{R}^*\Hom_{R}(U,X))$ for all $U\in$ $^{*}\mathcal{D}_{\square}(R)$ and $\H(U)\neq0$.
\item[(3)]$n\geq-\inf X$ and $^*\Ext^{n+1}_R(R/J,X)=0$ for every homogeneous ideal $J$ of
$R$.
\item[(4)]$n\geq-\inf X$ and for any (resp. some)
$^*$injective resolution $I$ of $X$, the graded $R$-module
$\Ker(\partial_{-n}:I_{-n}\to I_{-n-1})$ is $^*$injective.
\end{itemize}
Moreover the following hold:
\begin{align*}
^*\id_{R}X =& \sup\{j\in \mathbb{Z}| ^*\Ext^j_R(R/J,X)\neq0\text{
for some homogeneous ideal }J\}\\
=& \sup\{-\sup(U)-\inf(\mathbf{R}^*\Hom_{R}(U,X))|U\ncong0\text{ in
}^{*}\mathcal{D}_{\square}(R) \}.
\end{align*}
\end{thm}

\begin{proof} $(1)\Rightarrow(2)$ Let $t:=\sup U$ and $I$ be an $^*$injective resolution of
$X$, such that, for all $i<-n$, $I_i=0$. Then we have
$$
^*\Ext^i_R(U,X)\cong\H_{-i}(^*\Hom_R(U,I)).
$$
Since $^*\Hom_R(U,I)_{-i}=0$ for $-i<-n-t$, the assertion follows.

$(2)\Rightarrow(3)$ It is trivial that $^*\Ext^{n+1}_R(R/J,X)=0$ for
every homogeneous ideal $J$ of $R$. For the second assertion let
$U=R$ in (2). So that $\Ext^i_R(R,X)=\!^*\Ext^i_R(R,X)=0$ for $i>n$.
Now by \cite[Lemma 1.9(b)]{AF}, we have $\H_{-i}(X)=0$ for $-i<-n$.
This means that $n\geq-\inf X$.

$(3)\Rightarrow(4)$ By hypothesis of (4) $\H_i(I)=0$ for $i<-n$.
Thus the complex
$$
\cdots\to0\to0\to I_{-n}\to I_{-n-1}\to\cdots\to I_i\to
I_{i-1}\to\cdots,
$$
gives an $^*$injective resolution of $\Ker\partial_{-n}$. In
particular
$$
\!^*\Ext^1_R(R/J,\Ker\partial_{-n})=\H_{-n-1}\!^*\Hom_R(R/J,I)=\!^*\Ext^{n+1}_R(R/J,X)=0
$$
for every homogeneous ideal $J$ of $R$. Thus $\Ker\partial_{-n}$ is
$^*$injective by \cite[Corollary 4.3]{FF}.

$(4)\Rightarrow(1)$ Let $I$ be any $^*$injective resolution of $X$.
By (5) we have $\Ker\partial_{-n}$ is $^*$injective. Thus
$\!^*\id_RX<-n$ by definition.

The next two equalities are trivial.
\end{proof}

For a local ring $(R,\fm,k)$ and for an $R$-complex $X$ and $i\in
\mathbb{Z}$ the $i$th Bass number and Betti number of $X$ are
defined respectively by $
\mu^{i}_{R}(X):=\dim_{k}\H_{-i}(\uhom_{R}(k,X))$ and
$\beta_i^R(X):=\dim_{k}\H_i(k\otimes_R^\mathbf{L}X).$ It is
well-known that for $X\in{\mathcal D}_{\sqsubset} (R)$ one has (cf.
\cite[Proposition 5.3.I]{AF})
$$\id_{R}X=\sup\{m\in\mathbb{Z}|\exists\fp\in\Spec(R);\mu^m_{R_{\fp}}(X_{\fp})\neq0\}.$$

As a graded analogue we have:

\begin{prop}\label{lem id&mu}
For $X\in$ $^{*}\mathcal{D}_{\sqsubset}(R)$ we have the following
equality
$$^*\id_{R}X=\sup\{m\in\mathbb{Z}|\exists\fp\in ^*\Spec(R);\mu^m_{R_{\fp}}(X_{\fp})\neq0\}.$$
\end{prop}

\begin{proof} The argument is the same as proof of \cite[Proposition
5.3.I]{AF} with some changes. Denote the supremum by $i$. By Theorem
\ref{iid}, we have $\!^*\id_RX\geq i$. Hence the equality holds if
$i=\infty$. Thus assume that $i$ is finite. By Theorem \ref{iid} we
have to show that if $\!^*\Ext^j_R(M,X)\neq0$ for some finitely
generated graded $R$-module $M$, then $j\leq i$; this implies that
$\!^*\id_RX\leq i$. The elements of $Ass(M)$ are homogeneous prime
ideals. Thus we have a filtration $0=M_{0}\subset
 M_{1}\subset\cdots\subset M_t=M$ of graded submodules of $M$ such that
for each $i$ we have $M_{i}/M_{i-1}\cong  R/\fp_{i}$ with
$\fp_{i}\in\Supp M$ and is homogeneous. From the long exact sequence
of $\!^*\Ext^j_R(-,X)\neq0$ we have the set
$$
\{\fq\in\!\Spec(R)|\text{ there is an }h\geq j\text{ such that
}\!^*\Ext^h_R(R/\fq,X)\neq0\},
$$
is not empty. Let $\fp$ maximal in this set, and for a homogeneous
$x\in R\backslash\fp$ consider the exact sequence
$$
0\to R/\fp\stackrel{x}{\to} R/\fp\to R/(\fp+Rx)\to0.
$$
It induces an exact sequence
$$
\!^*\Ext^h_R(R/(\fp+Rx),X)\to\!^*\Ext^h_R(R/\fp,X)\stackrel{x}{\to}\!^*\Ext^h_R(R/\fp,X)\to\!^*\Ext^{h+1}_R(R/(\fp+Rx),X)
$$
in which the left-hand term is trivial because of the maximality of
$\fp$. Thus
$\!^*\Ext^h_R(R/\fp,X)\stackrel{x}{\to}\!^*\Ext^h_R(R/\fp,X)$ is
injective for all homogeneous elements $x\in R\backslash\fp$, hence
so is the homogeneous localization homomorphism
$\!^*\Ext^h_R(R/\fp,X)\to\!^*\Ext^h_R(R/\fp,X)_{(\fp)}$. Thus the
free $R_{(\fp)}/\fp R_{(\fp)}$-module
$\!^*\Ext^h_R(R/\fp,X)_{(\fp)}$ is nonzero. Consequently
$$
(\!^*\Ext^h_R(R/\fp,X)_{(\fp)})_{\fp
R_{(\fp)}}\cong\!^*\Ext^h_{R_{\fp}}(R_{\fp}/\fp R_{\fp},X_{\fp})
$$
is nonzero. This implies that $j\leq h\leq i$.
\end{proof}

\begin{rem}\label{pd} (1) A graded module is called $^*$projective if it is a projective
object in the category of graded modules. By \cite[Proposition
3.1]{FF} the $^*$projective graded $R$-modules coincide with
projective $R$-modules. The projective dimension of a graded module
$M$ in the category of graded modules, is denoted by $^*\pd_RM$ (cf.
\cite{FF}). Let $n\in \mathbb{Z}$. A homologically right bounded
complex of graded modules $X$, is said to have $^*$projective
dimension at most $n$, denoted by $^*\pd_RX\leq n$, if there exists
a $^*$projective resolution $P\to X$, such that $P_i=0$ for $i>n$.
If $^*\pd_RX\leq n$ holds, but $^*\pd_RX\leq n-1$ does not, we write
$^*\pd_RX=n$. If $^*\pd_RX\leq n$ for all $n\in \mathbb{Z}$ we write
$^*\pd_RX=-\infty$. If $^*\pd_RX\leq n$ for no $n\in \mathbb{Z}$ we
write $^*\pd_RX=\infty$.

(2) For $X\in\!^{*}\mathcal{D}_{\sqsupset}(R)$ by the same method as
in \cite[Theorem 2.4.P and Corollary 2.5.P]{AF} we have
\begin{align*}
^*\pd_{R}X =& \sup\{j\in \mathbb{Z}| ^*\Ext^j_R(X,N)\neq0\text{
for some graded }R\text{-module }N\}\\
=& \sup\{\inf(U)-\inf(\mathbf{R}^*\Hom_{R}(X,U))|U\ncong0\text{ in
}^{*}\mathcal{D}_{\square}(R) \}.
\end{align*}

(3) It is easy to see that for
$X\in\!^{*}\mathcal{D}_{\sqsupset}(R)$, we have
$\!^*\pd_RX\leq\pd_RX$.
\end{rem}

The proof of the following proposition is easy so we omit it (see
\cite[Theorem 1.5.9]{BH}). Let $J$ be an ideal of the graded ring
$R$. Then the graded ideal $J^*$ is denoted to the ideal generated
by all homogeneous elements of $J$. It is well-known that if $\fp$
is a prime ideal of $R$, then $\fp^*$ is a homogeneous prime ideal
of $R$ by \cite[Lemma 1.5.6]{BH}.

\begin{prop}\label{mu}
Assume that $X\in\,^{*}{\mathcal D}_{\square} (R)$ and $\fp$ is a
non homogeneous prime ideal in $R$. Then
$\mu^{i+1}_{R_{\fp}}(X_{\fp})= \mu^{i}_{R_{\fp^*}}(X_{\fp^*})$ and
$\beta_{i}^{R_{\fp}}(X_{\fp})=\beta_{i}^{R_{\fp^{*}}}(X_{\fp^{*}})$for
any integer $i\geq 0$.
\end{prop}

\begin{cor}\label{cor depth}
Let $X\in\,^*{\mathcal D}_{\square} (R)$ and $\fp$ be a
non-homogeneous  prime ideal in $R$. Then
$$\depth X_{\fp}=\depth X_{\fp^{*}}+1.$$
\end{cor}

\begin{proof} Using Proposition \ref{mu}, we can assume that both $\depth X_{\fp}$ and $\depth
X_{\fp^{*}}$ are finite. So the equality follows from the fact that
over a local ring $(R,\fm,k)$ we have $\depth_R X=\inf\{i\in
\mathbb{Z}|\mu^{i}_{R}(X)\neq0\}$.
\end{proof}


Foxby in \cite{F79} defined the \emph{small support} of a
homologically right bounded complex $X$ over a Noetherian ring $R$,
denoted by $\mbox{supp}_RX$, as
$$\supp_{R}X=\{\fp\in \Spec R|\exists m \in
\mathbb{Z}: \beta_{m}^{R_{\fp}}(X_{\fp})\neq 0\}.$$ Let
$\,^*\supp_{R} X$ be a subset of $\supp_{R}X$ consisting of
homogeneous prime ideals of $\supp_{R}X$. Then from Proposition
\ref{mu} we see that $\fp\in\supp_{R}X$ if and only if
$\fp^*\in\,^*\supp_{R} X$. Also using \cite[Proposition 2.8]{F79}
and Corollary \ref{cor depth} (or directly from Proposition
\ref{mu}) we have
$$\width_{R_{\fp}}X_{\fp}<\infty\Leftrightarrow\width_{R_{\fp^*}}X_{\fp^*}<\infty.$$

\begin{prop}\label{pro width}
Assume that $X\in\,^{*}{\mathcal D}_{\square} (R)$, and $\fp$ is a
non homogeneous prime ideal in $R$. Then
$$\width_{R_{\fp}}X_{\fp}=\width_{R_{\fp^*}} X_{\fp^*}$$
\end{prop}

\begin{proof} We can assume
that both $\width X_{\fp}$ and $\width X_{\fp^{*}}$ are finite
numbers. And the argument is dual to the proof of \cite[Theorem
1.5.9]{BH}.
\end{proof}

The ungraded version of the following theorem was proved for modules
by Chouinard \cite[Corollary 3.1]{Ch} and extended to complexes by
Yassemi \cite[Theorem 2.10]{Y}.

\begin{thm}\label{thm id}  Let $X\in$ $\,^{*}\mathcal{D}_{\square}(R)$. If $\,^*\id_RX<\infty$ then
$$
^*\id_RX=\sup \{\depth R_{\fp}-\width
X_{\fp}|\fp\in\,^{*}\!\Spec(R)\}.
$$
\end{thm}

\begin{proof} We have the following computations
\begin{align*}
^*\id_RX=& \sup\{m\in \mathbb{Z}|\exists \fp\in \,^*\Spec(R):\mu^{m}_{R_{\fp}}(M_{\fp})\neq 0\} \\[1ex]
= & \sup\{m\in \mathbb{Z}|\exists \fp\in \,^*\!\Spec(R):
\H_m(\uhom_{R_{\fp}}(\kappa(p),M_{\fp}))\neq 0\}\\[1ex]
= & \sup\{-\inf\uhom_{R_{\fp}}(\kappa(\fp),M_{\fp})|\fp\in\,^*\!\Spec(R)\}\\[1ex]
= & \sup\{\depth R_{\fp}-\width_{R_{\fp}}M_{\fp}|\fp\in
\,^*\!\Spec(R)\}.
\end{align*}
The first equality is by Proposition \ref{lem id&mu} and the last
one is by \cite[Lemma 2.6(a)]{Y}.
\end{proof}

The following corollary was already known for graded modules in
\cite[Corollary 4.12]{FF}.

\begin{cor}\label{ineq}
For every $X\in \,^{*}\mathcal{D}_{\square}(R)$, we have
$$\,^{*}\id_RX\leq\id_RX\leq\,^{*}\id_RX+1.$$
\end{cor}

\begin{proof}
First of all note that by proposition \ref{mu}, $\id_RX<\infty$ if
and only if $\,^*\id_RX<\infty$. The first inequality is clear by
Theorem \ref{thm id} and \cite[Theorem 2.10]{Y}. For the second one
let $\fp \in \Spec R$ be such that $ \id_RX=\depth
R_{\fp}-\width_{R_{\fp}}M_{\fp}$ by \cite[Theorem 2.10]{Y}. By
Corollary \ref{cor depth} and Proposition \ref{pro width} we have
$$\depth R_{\fp}-\width_{R_{\fp}}M_{\fp}\leq\depth
R_{\fp^*}-\width_{R_{\fp^*}}M_{\fp^*}+1\leq\,^{*}\id_RX+1,$$ where
the second inequality holds by Theorem \ref{thm id}.
\end{proof}

Here we define the $\!^*$dualizing complex for a graded ring and
prove some related results that we need in the next section.

\begin{defn}
A $^*$dualizing complex for a graded ring $R$ is a homologically
finite and bounded complex $D$, of graded $R$-modules, such that
$\,^{*}\id_RD<\infty$ and the homothety morphism $\psi :R\to
\textbf{R} \,^{*}\Hom_{R}(D,D)$ is invertible in $\,^*\mathcal
{D}(R)$.
\end{defn}

\begin{cor}
Any $^*$dualizing complex for $R$ is a dualizing complex for $R$.
\end{cor}

The proof of the following lemma is the same as \cite[Chapter V,
Proposition 3.4]{Ha}.

\begin{lem}
Let $(R,\fm,k)$ be a $^*$local ring and that $D$ is a $^*$dualizing
complex for $R$. Then there exists an integer $t$ such that
$\H^t(\textbf{R} \,^{*}\Hom_{R}(k,D))\cong k$ and $\H^i(\textbf{R}
\,^{*}\Hom_{R}(k,D))=0$ for $i\neq t$.
\end{lem}

Assume that $(R,\fm)$ is a $^*$local ring. A $^*$dualizing complex
$D$ is said to be \emph{normalized $^*$dualizing complex}, if $t=0$
in the lemma. It is easy to see that a suitable shift of any
$^*$dualizing complex is a normalized one. Also using \cite[Chapter
V, Proposition 3.4]{Ha} we see that if $D$ is a normalized
$^*$dualizing complex for $(R,\fm)$, then $D_{\fm}$ is a normalized
dualizing complex for $R_{\fm}$.

\begin{lem}\label{D}
Let $(R,\fm,k)$ be a $^*$local ring and that $D$ is a normalized
$^*$dualizing complex for $R$. Then there exists a natural
functorial isomorphism on the category of graded modules of finite
length to itself
$$\phi:\H^0(\textbf{R} \,^{*}\Hom_{R}(-,D))\to\,^{*}\Hom_R(-,\!^*\E_R(k)),$$
where $\!^*E_R(k)$ is the $^*$injective envelope of $k$ over $R$.
\end{lem}

\begin{proof}
Since $D$ is a normalized $^*$dualizing complex for $R$,
$T:=\H^0(\textbf{R} \,^{*}\Hom_{R}(-,D))$ is an additive
contravariant exact functor from the category of graded modules of
finite length to itself. Let $M$ be a graded $R$-module and $m\in M$
is homogeneous element of degree $\alpha$. Then
$\epsilon_m:R(-\alpha) \to M$ is a homogeneous morphism which sends
$1$ into $m$. Thus we have a homogeneous morphism $\phi(M):T(M)\to
\!^*\Hom_R(M,T(R))$ which sends a homogeneous element $x\in T(M)$ to
a morphism $f_x\in ^*\Hom_R(M,T(R))$ such that
$f_x(m)=T(\epsilon_m)(x)$ for every homogeneous element $m\in M$. It
is easy to see that it is functorial on $M$. Thus we showed that
there is a natural functorial morphism $\phi:T\to
\!^*\Hom_R(-,T(R))$. Therefore by the same method of \cite[Lemma 4.4
and Propositions 4.5]{G}, there is a functorial isomorphism
$$\phi:\H^0(\textbf{R}
\,^{*}\Hom_{R}(-,D))\to\,^{*}\Hom_R(-,\,^*\displaystyle\lim_{\longrightarrow}T(R/\fm^n)),$$
from the category of graded modules of finite length to itself.
Using the technique of proof of \cite[Proposition 4.7]{G} in
conjunction with \cite[Corollary 4.3]{FF}, we see that
$\,^*\displaystyle\lim_{\longrightarrow}T(R/\fm^n)$ is an
$^*$injective $R$-module. Since $D$ is a normalized $^*$dualizing
complex for $R$ we have
$$\,^{*}\Hom_R(k,\,^*\displaystyle\lim_{\longrightarrow}T(R/\fm^n))\cong\H^0(\textbf{R}
\,^{*}\Hom_{R}(k,D))\cong k.$$ Thus in particular we can embed $k$
to $\,^*\displaystyle\lim_{\longrightarrow}T(R/\fm^n)$. To show that
$\,^*\displaystyle\lim_{\longrightarrow}T(R/\fm^n)$ is an
$^*$essential extension of $k$, let $Q$ be a graded submodule of
$\,^*\displaystyle\lim_{\longrightarrow}T(R/\fm^n)$ such that $k\cap
Q=0$. Then $^*\Hom_R(k,Q)$ can be embed in
$$^*\Hom_R(k,\,^*\displaystyle\lim_{\longrightarrow}T(R/\fm^n))\cong
k.$$ Therefore $^*\Hom_R(k,Q)=0$. On the other hand
$\Ass(T(R/\fm^n))$ is in $V(\fm)$ for each $n\in \mathbb{N}$. Now by
\cite[Proposition 2.1]{SS}, the fact that each prime ideal of
$\Ass(\,^*\displaystyle\lim_{\longrightarrow}T(R/\fm^n))$ is the
annihilator of a homogeneous element \cite[Lemma 1.5.6]{BH}, and the
definition of $\,^*\displaystyle\lim_{\longrightarrow}$, we have
$$\Ass(\,^*\displaystyle\lim_{\longrightarrow}T(R/\fm^n))\subseteq
\bigcup_{n\in \mathbb{N}} \Ass(T(R/\fm^n))\subseteq V(\fm).$$
Consequently $Q$ has support in $V(\fm)$, so that $Q=0$. Therefore
$\,^*\displaystyle\lim_{\longrightarrow}T(R/\fm^n)\cong
\!^*\E_R(k)$.
\end{proof}

Let $\fa$ be an ideal of $R$. The right derived \emph{local
cohomology functor} with support in $\fa$ is denoted by
$\mathbf{R}\Gamma_{\fa}(-)$. Its right adjoint,
$\mathbf{L}\Lambda^{\fa}(-)$, is the left derived \emph{local
homology functor} with support in $\fa$ (see \cite{Fr} for detail).

Now we have the following proposition, which its proof use Lemma
\ref{D}, and the argument is the same as \cite[Chapter V,
Proposition 6.1]{Ha}.

\begin{prop}\label{e}
Let $(R,\fm,k)$ be a $^*$local ring and that $D$ be a normalized
$^*$dualizing complex for $R$. Then ${\mathbf
R}\Gamma_{\fm}(D)\simeq\!^*\E_R(k)$.
\end{prop}

\section{$^*$Gorenstein injective dimension}

In this section we introduce the concept of $\!^*$Gorenstein
injective dimension of complexes and we derive its main properties.
In particular we prove a Chouinard's like formula for this
dimension, and compare it with the usual Gorenstein injective
dimension.
\begin{defn} \label{defn Gid}
A graded $R$-module $N$ is called $\,^{*}$Gorenstein injective, if
there exists an acyclic complex $\textbf{I}$ of $\,^{*}$injective
$R$-modules and homogeneous homomorphisms such that $M\cong
\Ker(I^{0}\to I^{1})$ and for every $\,^{*}$injective module $E$,
the complex $\,^*\Hom_{R}(E,\textbf{I})$ is exact.
\end{defn}

It is clear that every $^*$injective $R$-module is
$\,^{*}$Gorenstein injective. So that every $Y\in
\,^*\mathcal{D}_{\sqsubset}(R)$ has a $\,^{*}$Gorenstein injective
resolution. The $^{*}$Gorenstein injective dimension of $Y\in
\,^*\mathcal{D}_{\sqsubset}(R)$ denoted by $\,^{*}\Gid_{R}Y$, is
define as:
$$\,^{*}\Gid_{R}Y:=\inf\left\{\sup\{\ell\in
Z|B_{-\ell}\neq0\}\bigg| \begin{array}{l} B_{\ell} \text{ is }\, ^*
\text{Gorenstein injective and}
\\ B\in\,^{*}\mathcal{D}_{\sqsubset}(R) \text{ is isomorphic to } Y \end{array} \right\}.$$




By a careful revision of the proof of dual version of \cite[Theorems
2.5 and 2.20]{H} we have the following two results.

\begin{thm}\label{thm resolving}
The class of $^*$Gorenstein injective $R$-modules is $^*$injectively
resolving, that is for any short exact sequence $0 \to X'\to X \to
X''\to 0$ with $X'$ $^*$Gorenstein injective $R$-module, $X''$ is
 $^*$Gorenstein injective if and only if $X$ is $^*$Gorenstein
 injective.
\end{thm}



\begin{prop}\label{thm Gid}
Let $N$ be a graded $R$-module with finite $^*$Gorenstein injective
dimension, and let $n$ be an integer. Then the following conditions
are equivalent:
\begin{itemize}
\item[(1)] $^*\!\Gid_RN\leq n$.
\item[(2)] $^*\!\Ext^i_R(L,N)=0$ for all $i>n$, and all $R$-modules $L$ with finite $^*\id _RL$.

\item[(3)] $^*\!\Ext^{i}_{R}(I,N)=0$ for all $i>n$, and all $^*$injective
$R$-modules $I$.

\item[(4)] For every exact sequence $0\rightarrow N\rightarrow
H^{0}\rightarrow\cdots\rightarrow H^{n-1}\rightarrow C^n\rightarrow
0$ where $H^0,\cdots,H^{n-1}$ are $^*$Gorenstein injectives, then
$C^n$ is also $^*$Gorenstein injective.
\end{itemize}

Consequently, the $^*$Gorenstein injective dimension of $M$ is
determined by the formulas:
\begin{align*}
^*\Gid_RN=& \sup\{i\in \mathbb{N}_0|\,^*\Ext^i_R(L,N)\neq0 \text{
for
some graded } R\text{-module } L \text{ with finite}\, ^*\id _RL\} \\[1ex]
= & \sup\{i\in \mathbb{N}_0|\,^*\Ext^i_R(I,N)\neq0 \text{ for some
graded } ^*\text{injective module } I\}.
\end{align*}
\end{prop}

The ungraded version of the following theorem is in \cite[Theorem
2.8]{CFH}. The proof uses Proposition \ref{quasi-iso} and the same
technique of proof of \cite[Theorem 2.8]{CFH}.

\begin{thm}\label{iso}
Let $V\stackrel{\simeq}\longrightarrow W$ be a quasiisomorphism
between complexes of graded $R$-modules, where each module in $V$
and $W$ has finite $^*$projective dimension or finite $^*$injective
dimension. If $B\in\, ^*\mathcal{C}_{\sqsubset}(R)$ is a complex of
$^*$Gorenstein injective modules, then the induced morphism
$$
^*\Hom_R(W,B)\to \,^*\Hom_R(V,B)
$$
is a quasiisomorphism under each of the next two condition:
\begin{itemize}
\item[(a)] $V, W\in \,^*\mathcal{C}_{\sqsupset}(R)$, or
\item[(b)] $V, W\in \,^*\mathcal{C}_{\sqsubset}(R)$.
\end{itemize}
\end{thm}

\begin{cor}\label{2.12}
Assume that $Y\simeq B$ where $B\in$ $^*\mathcal{C}_{\sqsubset}(R)$
is a complex of $^*$Gorenstein injective modules. If $U\simeq V$,
where $V\in$ $^*\mathcal{C}_{\sqsupset}(R)$ is a complex in which
each module has finite $^*$projective dimension or finite
$^*$injective dimension, then
$$\mathbf{R}^*\!\Hom_{R}(U, Y)\simeq \,^*\Hom_{R}(V, B).$$
\end{cor}

\begin{proof} Is the same as \cite[Corollary 2.10]{CFH} using Theorem
\ref{iso}(a).
\end{proof}

The ungraded version of the following theorem is contained in
\cite[Theorem 3.3]{CFH}, and its proof is dual of \cite[Theorem
3.1]{CFH}. We present the proof of (3)$\Rightarrow$(4) for later
use.


\begin{thm}\label{thm Gid N}
Let $Y\in$ $^{*}\mathcal{D}_{\sqsubset}(R)$ be a complex of finite
$^{*}$Gorenstein injective dimension. For $n\in \mathbb{Z}$ the
following are equivalent:
\begin{itemize}
\item[(1)]$^{*}\Gid_{R}Y \leq n.$
\item[(2)] $n\geq -\sup U-\inf \mathbf{R}^{*}\!\Hom_{R}(U,Y)$ for all
$U\in$ $^{*}\mathcal{D}_{\square}(R)$ of finite $^{*}$projective or
finite $^{*}$injective dimension with $H(U)\neq 0$.
\item[(3)]$n\geq -\inf \mathbf{R}^{*}\!\Hom_{R}(J,Y)$ for all $^{*}$injective
$R$-modules $J$.
\item[(4)]$n\geq -\inf Y$ and for any left-bounded complex $B\simeq Y$ of
$^{*}$Gorenstein injective modules, the $\Ker(B_{-n}\rightarrow
B_{-(n+1)})$ is a $^*$Gorenstein injective module.
\end{itemize}
Moreover the following hold:
\begin{align*}
^*\Gid_{R}Y=& \sup\{-\sup(U)-\inf
\mathbf{R}^*\Hom_{R}(U,Y)| ^*\id_{R}(U)<\infty \text{ and } H(U)\neq 0\}\\
=&\sup\{-\inf \mathbf{R}^{*}\!\Hom_{R}(J,Y)| J \text{ is }
^{*}\text{injective}\}.
\end{align*}
\end{thm}

\begin{proof}
(2)$\Rightarrow$(3) and (4)$\Rightarrow $(1) are clear.
$(1)\Rightarrow(2)$ is dual of $(1)\Rightarrow(2)$ in \cite[Theorem
3.1]{CFH}.



(3)$\Rightarrow$(4): To establish $n\geq -\inf Y$, it is sufficient
to show that
$$
\sup\{-\inf \mathbf{R}^{*}\!\Hom_{R}(J,Y) |J \text{ is }
^{*}\text{injective}\}\geq -\inf Y. \quad (*)
$$
By assumption $g:=\!^*\Gid_R(Y)$ is finite, so $Y\simeq B$ for some
complex of $^*$Gorenstein injective modules:
$$B= 0\to B_{s}\to B_{s-1}\to \cdots\to B_{-g+1}\to B_{-g}\to 0.$$
Now it is clear that $-g\leq \inf Y$. By Lemma \ref{2.12}, for any
$^*$injective module $J$, the complex $\!^*\Hom_{R}(J,B)$ is
isomorphic to $\mathbf{R}^*\Hom_{R}(J,Y)$ in $^*\mathcal{D}(R)$. If
$g=-\inf Y$ then the differential $\partial_{-g+1}:B_{-g+1}\to
B_{-g}$ is not surjective. Now by the definition of $^*$Gorenstein
injective modules there exists an $^*$injective module $J$ such that
$J\to B_{-g}$ is surjective. Notice that the differential
$^*\Hom_R(J,\partial_{-g+1})$ in the complex $^*\Hom_{R}(J,B)$ is
not surjective, for otherwise, for any $\phi\in
^*\Hom_{R}(J,B_{-g})$ there exists a $\psi\in
^*\Hom_{R}(J,B_{-g+1})$ such that $\phi=\partial_{-g+1}\psi$. This
implies that $\partial_{-g+1}$ is surjective which is a
contradiction. Therefore $^*\Hom_{R}(J,B)$ has nonzero homology in
degree $-g=\inf Y$. This gives $(*)$. Next assume that $g>-\inf
Y=-t$ and consider the exact sequence
$$
B: 0\to Z_{t}^{B}\to B_{t}\to B_{t-1} \to\cdots\to B_{-g+1}\to
B_{-g}\to 0.\quad (**)
$$
It shows that $^*\Gid_R Z^{B}_{t}\leq t+g$ and it is not difficult
to see that the equality must hold. For otherwise $^*\Gid_RY<g$ and
by Proposition \ref{thm Gid},
$\H_{-g}(\mathbf{R}^*\Hom_{R}(J,Y))\cong\!^*\Ext^{t+g}(J,Z^{B}_{t})\neq
0$ for some $^*$injective module $J$. Therefore $(*)$ follows, which
gives the inequality $n\geq -\inf Y$.

To prove the second part of (4) let $B$ be a left-bounded complex of
$^*$Gorenstein injective modules such that $B\simeq Y$. By
assumption $^*\Gid_RY$ is finite, so there exists a bounded complex
$\tilde{B}$ of $^*$Gorenstein injective modules such that
$\tilde{B}\simeq Y$. Since $n\geq -\inf Y=-\inf \tilde{B}$, the
kernel $Z^{\tilde{B}}_{n}$ fits in an exact sequence
$$0\to Z^{\tilde{B}}_{-n}\to \tilde{B}_{-n}\to
\tilde{B}_{-n-1}\to\cdots\to\tilde{B}_{t}\to0.$$ By Proposition
\ref{thm Gid} and the isomorphism
$^*\Ext^{i}(J,Z^{\tilde{B}}_{-n})\simeq
\H_{-(n+i)}(\mathbf{R}^*\Hom_{R}(J,Y))=0$ for $i>0$ we get that
$Z^{\tilde{B}}_{-n}$ is $^*$Gorenstein injective.

Now it is enough to prove that if $I$ and $B$ are left bounded
complexes of respectively $^*$injective and $^*$Gorenstein injective
modules and $B\simeq Y \simeq I$ then the kernel $Z^{I}_{-n}$ is
$^*$Gorenstein injective if and only if $Z^{B}_{-n}$ is so. Let $B$
and $I$ be such complexes. As $I$ consists of $^*$injectives by
\cite[Chapter I, Lemma 4.5]{Ha} there is a quasi isomorphism
$\pi:B\rightarrow I$ which induces a quasi isomorphism between the
complexes $\pi\supset_{n}:B\supset_{n}\rightarrow I\supset_{n}$. The
mapping cone
$$\mathcal{M}(\pi\supset_{n})=\cdots \to B_{n-1}\oplus I_{n}\to
B_{n}\oplus I_{n+1}\to B_{n+1}\oplus Z^{B}_{n}\to Z^{I}_{n}\to 0$$
is bounded exact such that all modules but the two right-most ones
are $^*$Gorenstein injective modules. It follows by Theorem \ref{thm
resolving} that $Z^{I}_{n}$ is $^*$Gorenstein injective if and only
if $B_{n+1}\oplus Z^{B}_{n}$ is so, which is tantamount to
$Z^{B}_{n}$ being $^*$Gorenstein injective.


The two equalities are immediate consequences of the equivalence of
(1)-(4).
\end{proof}

\begin{cor}\label{gid<id}
Let $Y\in$ $^{*}\mathcal{D}_{\sqsubset}(R)$. Then
$$^*\Gid_{R}Y\leq ^{*}\id_{R}Y,$$ with equality if $^*\id_{R}Y$ is
finite.
\end{cor}

\begin{proof} Is the same as \cite[Proposition 6.2.6]{C} using Theorem \ref{thm Gid N} and Corollary \ref{ineq}.
\end{proof}

Recall the \emph{finitistic injective dimension} of $R$ which
defined as
$$
\FID(R):=\sup\{\id_RM|M\text{ is an }R\text{-module with
}\id_RM<\infty\}.
$$
The \emph{finitistic Gorenstein injective dimension}
$\mbox{FGID}(R)$, \emph{finitistic $^*$injective dimension}
$^*\FID(R)$ and \emph{finitistic $^*$Gorenstein injective dimension}
$^*\mbox{FGID}(R)$ are define similarly. In \cite[Theorem 2.29]{H},
Holm proved that $\mbox{FGID}(R)=\FID(R)$. Also it is known that
over a commutative Noetherian ring $R$ we have $\FID(R)\leq\dim R$
by \cite[Corollary 5.5]{B} and \cite[II. Theorem 3.2.6]{RG}.

The following theorem is the graded version of Holm's result
\cite[Theorem 2.29]{H}. Its proof is dual of \cite[Theorem 2.28]{H}.
We give the sketch of proof.

\begin{thm} If $R$ is any graded ring $^*\FGID(R)=^*\FID(R)$.
\end{thm}

\begin{proof} Using Corollary \ref{gid<id} we have
$^*\FID(R)\leq\!^*\FGID(R)$. If $N$ is a graded $R$-module with
$0<\,^*\Gid_RN<\infty$, then the graded version of \cite[Theorem
2.15]{H} gives a graded $R$-module $C$ such that
$^*\id_RC=\,^*\Gid_RN-1$. Therefore $^*\FGID(R)\leq\,^*\FID(R)+1$.
Now for the reverse inequality $\!^*\FGID(R)\leq\,^*\FID(R)$, assume
that $0<\,^*\FGID(R)=m<\infty$. Pick a module $N$ with
$^*\Gid_RN=m$. Hence the same technique of \cite[Lemma 2.2]{SSH},
gives a graded module $T$ such that $^*\id_RT=m$.
\end{proof}

\begin{cor}\label{FID}
Let $Y\in$ $^{*}\mathcal{D}_{\sqsubset}(R)$ be a complex of finite
$^{*}$Gorenstein injective dimension. Then
$$
^*\Gid_{R}Y\leq \FID(R)-\inf Y.
$$
\end{cor}

\begin{proof} Note that by $(**)$ in the proof of Theorem \ref{thm Gid N} above
we have $\!^*\Gid_RY=\!^*\Gid_RM-\inf Y$ for some graded $R$-module
$M$. Thus $\!^*\Gid_RY\leq\!^*\mbox{FGID}(R)-\inf Y=\!^*\FID(R)-\inf
Y\leq\FID(R)-\inf Y$.
\end{proof}

Let $D$ be a $\!^*$dualizing complex of $R$. Then $D$ is a dualizing
complex of $R$, so we have the Bass category $B(R)$ with respect to
$D$ (cf. \cite[Page 237]{CFH}). It is known that for
$Y\in\mathcal{D}_{\sqsubset}(R)$, we have $Y\in $ $B(R)$ if and only
if $\Gid_{R}Y<\infty$, \cite[Theorem 4.4]{CFH}.

\begin{defn} Let $D$ be a $\!^*$dualizing complex
of $R$. The $^*$Bass category $^{*}B(R)$ with respect to $D$ is
define as:
$$^{*}B(R):=\left\{Y\in
^*\mathcal{D}_{\square}(R)\bigg|\begin{array}{l}\epsilon_{Y}:D\otimes_{R}^{\mathbf{L}}
\mathbf{R}^{*}\!\Hom_{R}(D, Y)\to Y \text{ is an iso-}
\\\text{morphism and } \mathbf{R}^{*}\Hom_{R}(D, Y)\in\!^*\mathcal{D}_{\square}(R)
\end{array} \right\}.$$
\end{defn}

It is easily seen that  a complex $Y\in ^*\mathcal{D}_{\square}(R)$
is in $B(R)$ if and only if is in $\!^*B(R)$.

\begin{thm}\label{B(R)}
Assume that $R$ admits a $\!^*$dualizing complex $D$. Then For
$Y\in$ $^{*}\mathcal{D}_{\sqsubset}(R)$ the following are
equivalent:
\begin{itemize}
\item[(1)] $Y\in\!^{*}B(R)$.
\item[(2)] $^{*}\Gid_{R}Y<\infty$.
\end{itemize}
\end{thm}

\begin{proof} It is dual to the proof of \cite[Theorem 4.1]{CFH}. In fact the proof uses that if $N$ is a
graded $R$-module satisfying both $N \in$ $^{*}B(R)$ and
$^*\Ext^{m}_{R}(J,N)=0$ for all integer $m>0$ and all
$^{*}$injective $R$-module $J$, then $N$ is $^{*}$Gorenstein
injective, which its proof is dual to \cite[Lemma 4.6]{CFH}.
\end{proof}

\begin{thm}\label{thm width&depyh}
Let $(R,\fm,k)$ be a $^*$local ring that admits a $^*$dualizing
complex $D$. For a complex $Y\in$ $^*\mathcal{D}_{\square}(R)$ of
finite $^*$Gorenstein injective dimension, we have
$$\width(\fm,Y)=\depth(\fm,R)+ \inf \textbf{R}^*\Hom_{R}(\!^*\E_R(k),Y).$$
\end{thm}
\begin{proof}
We follow the method of \cite[Theorem 6.5]{CFH}. By Theorem
\ref{B(R)}, $Y\in \!^*B(R)$; in particular $Y\simeq
D\utp_R\textbf{R}^*\Hom_{R}(D,Y)$. Furthermore, we can assume that
$D$ is a normalized $^*$dualizing complex, so that by Proposition
\ref{e} we have ${\mathbf R}\Gamma_{\fm}(D)\cong\!^*\E_R(k)$ and
that $D_{\fm}$ is a normalized dualizing complex for $R_{\fm}$. We
compute as follows:
\begin{align*}
\width(\fm,Y)=& \width(\fm,D\utp_R\textbf{R}^*\Hom_{R}(D,Y)) \\[1ex]
= & \width(\fm,D)+\width(\fm,\textbf{R}^*\Hom_{R}(D,Y)) \\[1ex]
= & \inf D_{\fm}+\inf\textbf{L}\Lambda^{\fm}\mathbf{R}\!^*\Hom_R(D,Y) \\[1ex]
= & \depth R_{\fm}+\inf\textbf{R}^*\Hom_{R}({\mathbf R}\Gamma_{\fm}(D),Y)\\[1ex]
= & \depth(\fm,R)+\inf\textbf{R}^*\Hom_{R}(\!^*\E_R(k),Y).
\end{align*}
The second equality is by \cite[Theorem 2.4(b)]{Y}, the third one by
the fact that $D_{\fm}$ is homologically finite and by \cite[Theorem
2.11]{Fr}, and the forth one by \cite[2.6]{Fr}, and the fact that
$\textbf{R}^*\Hom_{R}(D,Y)=\uhom_{R}(D,Y)$.
\end{proof}

The ungraded version of the following result is in \cite[Proposition
5.5]{CFH}.

\begin{prop} \label{Gid_R_(p)}
Assume that $R$ admits a $\!^*$dualizing complex and let $Y\in$
$^*\mathcal{D}_{\square}(R)$. Then for any homogeneous prime ideal
$\fp\in R$ there is an inequality
$$^*\Gid _{R_{(\fp)}}Y_{(\fp)}\leq ^*\Gid_{R}Y.$$
\end{prop}

\begin{proof}
It is enough to show that if $N$ is a $\!^*$Gorenstein injective,
then $N_{(\fp)}$ is $\!^*$Gorenstein injective over $R_{(\fp)}$.
This is similar to the proof of \cite [Theorem 6.2.13]{C} using
Corollary \ref{FID}.
\end{proof}

The following proposition is the graded version of \cite[Lemma
2.1]{CS}. For part $(b)$ we follow the technique of \cite[Lemma
2.1]{CS}. We present the proof to give some hints for the graded
analogues. Before doing that we need a lemma.

\begin{lem}\label{comp}
Let $(R,\fm)$ be a $^*$local non-negatively graded ring. Then the
$\fm$-$^*$adic completion $^*\widehat{R}$ of $R$, is a
$^*$faithfully flat $R$-module, that is $^*\widehat{R}$ is $R$-flat
and for any graded $R$-module $M$, $M=0$ if and only if
$M\otimes_R\,^*\widehat{R}= 0$.
\end{lem}

\begin{proof} It is well known that $^*\widehat{R}$ is a flat
$R$-module by \cite[Corollary 3.3]{FF}. Now let $M$ be a graded
$R$-module such that $M\otimes_R\,^*\widehat{R}=0$. Note that
$\fm=\fm_0\oplus R_1\oplus R_2\oplus\cdots$, where $\fm_0$ is the
unique maximal ideal of $R_0$. So that
$(^*\widehat{R})_0=\displaystyle\lim_{\longleftarrow}(R/\fm^n)_0=\displaystyle\lim_{\longleftarrow}R_0/\fm_0^n=\widehat{R_0}$,
where $\widehat{R_0}$ is the $\fm_0$-adic completion of $R_0$. Let
$t$ be an integer and set $M_{(t)}:=M_t\oplus M_{t+1}\oplus
M_{t+2}\oplus\cdots$, which is a graded submodule of $M$. Hence
$M_{(t)}\otimes_R\,^*\widehat{R}= 0$. Therefore
$(M_{(t)}\otimes_R\,^*\widehat{R})_t=0$. Since
$(M_{(t)}\otimes_R\,^*\widehat{R})_t$ is generated as a
$\mathbb{Z}$-module by $x\otimes r$ for $x\in M_t$ and
$r\in(^*\widehat{R})_0=\widehat{R_0}$ and $t$ is arbitrary, we see
that $M\otimes_R\widehat{R_0}=0$. On the other hand we have
$R_0\otimes_R\widehat{R_0}=R_0\otimes_R(R_0\otimes_{R_0}\widehat{R_0})=(R_0\otimes_RR_0)\otimes_{R_0}\widehat{R_0}=
(R_0\otimes_{R_0}R_0)\otimes_{R_0}\widehat{R_0}=R_0\otimes_{R_0}\widehat{R_0}=\widehat{R_0}$.
Hence
$0=M\otimes_R\widehat{R_0}=(M\otimes_{R_0}R_0)\otimes_R\widehat{R_0}=M\otimes_{R_0}(R_0\otimes_R\widehat{R_0})
=M\otimes_{R_0}\widehat{R_0}$. Since $\widehat{R_0}$ is a faithfully
flat $R_0$-module, we get that $M=0$. This completes the proof.
\end{proof}

\begin{prop}\label{lem depth&widh}
Let $N$ be a $^*$Gorenstein injective $R$-module. Then under each of
the following conditions
\begin{itemize}
\item[(a)] $R$ admits a $^*$dualizing complex; or
\item[(b)] $R$ is a non-negatively graded ring,
\end{itemize}
one has
$$
\depth R_{\fp}\leqslant \width _{R_{\fp}} N_{\fp}
$$
for every $\fp$ in $^*\Spec R$, and equality holds if $\fp$ is a
maximal element in $^*\supp_RN$.
\end{prop}

\begin{proof} For part $(a)$ let $\fp$ be a homogenous prime ideal. Consider the $^*$local ring
$(S,\fn):=(R_{(\fp)}, \fp R_{(\fp)})$. By Proposition
\ref{Gid_R_(p)}, $B:=N_{(\fp)}$ is a $^*$Gorenstein injective
$S$-module. Thus
$$^*\Ext^{i}_S(\!^*\E_S(S/\fn),B)=0$$ for all $i>0$, that is
$\inf\textbf{R}^*\Hom_S(\!^*\E_S(S/\fn),B)\geq 0$. On the other hand
since $B_{\fn}\cong N_\fp$ and $S_{\fn}\cong R_\fp$, by Theorem
\ref{thm width&depyh} we have
$$\width N_{\fp}=\depth R_{\fp}+
\inf \textbf{R}^*\Hom_S(\!^*\E_S(S/\fn),B).$$ Thus $\depth
R_{\fp}\leqslant \width _{R_{\fp}} N_{\fp}$.

Now if $\width _{R_{\fp}} N_{\fp}$ is finite, we have
$\inf\textbf{R}^*\Hom_S(\!^*\E_S(S/\fn),B)<\infty$. Using Lemma
\ref{2.12} we have
$$\textbf{R}^*\Hom_S(\!^*\E_S(S/\fn),B)\simeq ^*\Hom_S(\!^*\E_S(S/\fn),B).$$ Therefore the infimum must be zero. This
proves the second statement.

For $(b)$ assume that $\fp$ is a homogeneous prime ideal and $T$ is
a graded $R_{(\fp)}$-module with $\,^*\pd_{R_{(\fp)}}T<\infty$. A
standard dimension shifting argument shows that
$\,^*\Ext^i_{R_{(\fp)}}(T,N_{(\fp)})=0$ for all $i>0$. Set
$d=\depth(\fp R_{(\fp)},R_{(\fp)})$ and choose a homogeneous maximal
$R_{(\fp)}$-regular sequence ${\bf x}$ in $\fp R_{(\fp)}$ by
\cite[Proposition 1.5.11]{BH}. Since the $^*$projective dimension of
$R_{(\fp)}/({\bf x})$ is finite we have
\begin{align*}
0\leq & \inf\mathbf{R}\!^*\Hom_{R_{(\fp)}}(R_{(\fp)}/({\bf
x}),N_{(\fp)})\\ \leq & \width(\fp
R_{(\fp)},\mathbf{R}\!^*\Hom_{R_{(\fp)}}(R_{(\fp)}/({\bf
x}),N_{(\fp)}))\\ = & \width(\fp
R_{(\fp)},N_{(\fp)})-d=\width_{R_{\fp}}N_{\fp}-d,
\end{align*}
where the second inequality holds by \cite[4.3]{CFF}, and the first
equality follows from Proposition \ref{Y}.

Now let $\fp$ be a maximal element in $^*\supp_RN$. Set
$S=R_{(\fp)}$ which is a $^*$local ring with depth $d$, homogeneous
maximal ideal $\fn=\fp R_{(\fp)}$, $B=N_{(\fp)}$, and
$l=R_{(\fp)}/\fp R_{(\fp)}$. One has $\fn\in\supp_SB$ and by exactly
the same method of proof of \cite[Lemma 2.1]{CS} we have
$\,^*\Ext^i_S(T,B)=0=\,^*\Ext^i_S(E,B)$ for all $i>0$ and every
graded $S$-module $T$ with $\,^*\pd_ST$ finite and $E:=\!^*\E_S(l)$.

Let $K$ denote the Koszul complex on a homogeneous system of
generators for $\fn$. Since $\H_i(K\otimes_SE)$ are Artinian, and by
\cite[Corollary 1.6.13]{BH}, we have $\fn \H_i(K\otimes_SE)=0$, we
see that $\H_i(K\otimes_SE)$ are finitely generated. So there is a
resolution $L\stackrel{\simeq}\longrightarrow K\otimes_SE$ by
finitely generated free S-modules. Using the $^*$Hom-evaluation we
have
$$
K\otimes_S(E\utp_S\!^*\Hom_S(E,B))\simeq
L\otimes_S\!^*\Hom_S(E,B)\cong\!^*\Hom_S(\!^*\Hom_S(L,E),B).
$$
The free resolution above induces a quasiisomorphism $\alpha$ from
$\,^*\Hom_S(K\otimes_SE,E)\cong\!^*\Hom_S(K,\!^*\widehat{S})$ to
$\,^*\Hom_S(L,E)$. The mapping cone $C=\mathcal{M}(\alpha)$ is a
bounded complex of direct sums of $^*\widehat{S}$ and $E$. Thus
$\,^*\Ext^i_S(C_j,B)=0$ for all $i>0$ and all $j$. Hence, an
application of $\,^*\Hom_S(-,B)$ yields a quasiisomorphism
$$
\,^*\Hom_S(\alpha,B):\!^*\Hom_S(\!^*\Hom_S(L,E),B)\stackrel{\simeq}\longrightarrow\!^*\Hom_S(\!^*\Hom_S(K,\!^*\widehat{S}),B).
$$
The modules in the complex $\,^*\Hom_S(K,^*\widehat{S})$ are
Ext-orthogonal to the modules in the mapping cone of an injective
resolution $B\stackrel{\simeq}\longrightarrow H$. Therefore, one has
$$
\!^*\Hom_S(\!^*\Hom_S(K,\!^*\widehat{S}),B)\simeq\!^*\Hom_S(\!^*\Hom_S(K,\!^*\widehat{S}),H)
$$
by the graded analogue of \cite[Lemma 2.4]{CFH}. Now piece together
the last three quasiisomorphisms, and use $^*$Hom-evaluation to
obtain
$$
K\otimes_S(E\utp_S\!^*\Hom_S(E,B))\simeq
K\otimes_S\mathbf{R}^*\Hom_S(\!^*\widehat{S},B).
$$
Therefore by \cite[(4.2) and (4.11)]{CFF}, the complexes
$E\utp_S\!^*\Hom_S(E,B)$ and $\mathbf{R}^*\Hom_S(\!^*\widehat{S},B)$
have the same width. From \cite[Theorem 2.4(b)]{Y} and Proposition
\ref{Y} we have
$$
\width(\fn,E)+\width(\fn,\,^*\Hom_S(E,B))=\width(\fn,B).
$$
Now we can see that $\width(\fn,\,^*\Hom_S(E,B))=0$ (see proof of
\cite[Lemma 2.1]{CS}). Consequently
$\width_{R_{\fp}}N_{\fp}=\width(\fn,B)=\width(\fn,E)=\depth(\fn,\!^*\widehat{S})$
using Proposition \ref{matlis}. Now since $\!^*\widehat{S}$ is a
flat $S$-module we have
$$
\Ext^i_S(S/\fn,\!^*\widehat{S})\cong\Ext^i_S(S/\fn,S)\otimes_S\!^*\widehat{S}.
$$
Keep in mind that $\Ext^i_S(S/\fn,S)=\,^*\Ext^i_S(S/\fn,S)$ is a
graded $S$-module. Hence using Lemma \ref{comp} we have
$\Ext^i_S(S/\fn,\!^*\widehat{S})=0$ if and only if
$\Ext^i_S(S/\fn,S)=0$. Therefore
$\depth(\fn,\!^*\widehat{S})=\depth(\fn,S)=d$.
\end{proof}

\begin{lem}\label{lem pullback}
Let $X$ be a complex of graded $R$-modules. For any homogeneous
surjective homomorphism $i:Y_n \longrightarrow X_n$ of graded
$R$-modules there is a commutative diagram
\begin{displaymath}
\xymatrix {Y=\cdots \ar[r]& X_{n+2}\ar[r]^{\gamma} \ar[d]^{=}&
Y_{n+1}
\ar[r]^{\beta}\ar[d]^{i'}& Y_{n}\ar[r]^{\alpha}\ar[d]^{i} & X_{n-1} \ar[r]\ar[d]^{=} &\cdots\\
X=\cdots\ar[r]& X_{n+2}\ar[r]^{\gamma'}& X_{n+1} \ar[r]^{\beta'} &
X_{n}\ar[r]^{\alpha'} & X_{n-1} \ar[r]&\cdots}
\end{displaymath}
such that $Y$ is a complex of graded $R$-modules, that $i'$ is
surjective and $\Ker i\cong \Ker i'$ and the induced map
$\H(Y)\longrightarrow \H(X)$ is an isomorphism.
\end{lem}

\begin{proof}
Let $\alpha=\alpha'i$ and $\beta: Y_{n+1}\longrightarrow Y_{n}$ be
the pullback of $\beta'$ along $i$ thus
$$ Y_{n+1}=\{(x,y)|x\in X_{n+1}, y\in Y_n \text{ and } \beta'(x)=i(y)\}.$$
It is clear that $ Y_{n+1}$ is a graded $R$-module. Let
$i':Y_{n+1}\to X_{n+1}$ be a map defined by $i'(x,y)=x$ which is
seen to be surjective. Define $\gamma:X_{n+2}\to Y_{n+1}$ by
$\gamma(x)=(\gamma'(x),0)$ for every $x\in X_{n+2}$. It is clear
that $Y$ is a complex and $\Ker i\cong \Ker i'$. Therefore the
induced map $\H(Y)\longrightarrow \H(X)$ is an isomorphism.
\end{proof}

\begin{cor}\label{triangle}
Let $Y\in$ $^{*}\mathcal{D}_{\square}(R)$ such that $^*\Gid_R Y=g
>0$. Then there is an exact triangle $G\to I\to Y\to\Sigma G$  such
that $G$ is a $^*$Gorenstein injective module and $I$ is a complex
of graded $R$-modules such that $^*\id_R I=g$.
\end{cor}

\begin{proof}
Without loss of generality we can assume that $Y$ has the form
$$ 0\to I_0\to I_{-1} \to \cdots \to I_{-g+1}\to B_{-g}\to 0,$$
where the $I_j$s are $^*$injective and $B_{-g}$ is $^*$Gorenstein
injective. By definition $B_{-g}$ is homomorphic image of some
$\!^*$injective module $J_0$. Thus by Lemma \ref{lem pullback} we
have a commutative diagram
\begin{displaymath}
\xymatrix {0 \ar[r]& I_0\ar[r] \ar[d]& I_{-1}
\ar[r]\ar[d]& \cdots\ar[r] & I_{-g+2} \ar[r]\ar[d] &G_{-g+1}\ar[r]\ar[d]^{i'}&J_0\ar[r]\ar[d]^i&0\\
0 \ar[r]& I_0\ar[r] & I_{-1} \ar[r]& \cdots \ar[r] & I_{-g+2} \ar[r]
&I_{-g+1}\ar[r]\ar[d]&B_{-g}\ar[r]\ar[d] & 0,\\
&&&&&0&0}
\end{displaymath}
such that $\Ker i\cong\Ker i'$. Since $\Ker i$ is a $^*$Gorenstein
injective module, using Theorem \ref{thm resolving}, $G_{-g+1}$ is
also $^*$Gorenstein injective. By repeating this argument $g$ times
we get the desired exact triangle $G\to I\to Y\to\Sigma G$.
\end{proof}

The following equality extends Chouinard's formula \cite{Ch} and
\cite[Theorem 2.2]{CS} to the $^*$Gorenstein injective dimension.

\begin{thm}\label{thm Gid formula}\label{C}
Assume that $R$ admits a $\!^*$dualizing complex or $R$ is a
non-negatively graded ring. Let $Y\in$
$^{*}\mathcal{D}_{\square}(R)$ of finite $^*$Gorenstein injective
dimension. Then there is an equality
$$ ^*\Gid _{R}Y=\sup\{\depth R_{\fp}-\width_{R_{\fp}}Y_{\fp} |\fp\in ^*\Spec R \}.$$
\end{thm}

\begin{proof}
The argument is similar to \cite[Theorem 2.2]{CS}, just use
Proposition \ref{lem depth&widh} and Corollary \ref{triangle}.
\end{proof}

In the following corollary we compare the $^*$Gorenstein injective
dimension with the usual Gorenstein injective dimension.

\begin{cor}\label{ingid}
Let $Y\in$ $^{*}\mathcal{D}_{\square}(R)$.Then under each of the
following conditions
\begin{itemize}
\item[(a)] $R$ admits a $^*$dualizing complex; or
\item[(b)] $R$ is a non-negatively graded ring, and $^*\Gid_{R}Y$ and
$\Gid_{R}Y$ are finite,
\end{itemize}
one has
$$
^*\Gid_{R}Y\leq\Gid_{R}Y\leq ^{*}\Gid_{R}Y+1.
$$
\end{cor}

\begin{proof}
If $R$ admits a $^*$dualizing complex, then Theorem \ref{B(R)},
implies that $^{*}\Gid_{R}Y<\infty$, if and only if
$\Gid_{R}Y<\infty$. In each both cases the inequalities follow from
Theorem \ref{thm Gid formula} and \cite[Theorem 2.2]{CS}.
\end{proof}



\begin{thebibliography}{10}

\bibitem{AF} L. L. Avramov and H. B. Foxby, {\em
Homological dimensions of unbounded complexes}, J. Pure Appl.
Algebra. {\bf 71}, (1991), 129--155.

\bibitem{B} H. Bass, {\em
Injective dimension in Noetherian rings}, Trans. Amer. Math. Soc.,
{\bf 102}, (1962), 18--29.

\bibitem{BH} W. Bruns and J. Herzog, {\em Cohen-Macaulay rings},
Cambridge Studies in Advanced Mathematics. {\bf 39}, Cambridge
University Press, Cambridge, 1998.

\bibitem{Ch} Leo G. Chouinard II, {\em On finite weak and injective
dimension}, Proc. Amer. Math. Soc. {\bf 60}, (1976), 57--60.

\bibitem{C} L. W. Christensen, {\em Gorenstein dimensions}, Lecture Notes in Mathematics, {\bf 1747},
Springer-Verlag, Berlin, 2000.

\bibitem{CFF} L. W. Christensen, H. B. Foxby, and A. Frankild, {\em Restricted
homological dimensions and Cohen-Macaulayness}, J. Algebra, {\bf
251}, (2002), no. 1, 479--502.

\bibitem{CFH} L.W. Christensen, A. Frankild, and H. Holm {\em On Gorenestein
Projective, Injective and Flat dimensions - A Functorial description
with applications}, J. Algebra, {\bf 302}, (2006), 231-279.

\bibitem{CS} L. W. Christensen and S. Sather-Wagstaff, {\em Transfer of Gorenstein dimensions along ring homomorphisms},
J. Pure Appl. Algebra.  {\bf 214}, (2010), 982--989.


\bibitem{F} R. M. Fossum, {\em The structure of indecomposable injective modules}, Math. Scand.
{\bf 36}, (1975), 291--312.

\bibitem{FF} R. M. Fossume and H. B. Foxby, {\em The Category of Graded Modules},
Math. Scand. {\bf 35}, (1974), 288--300.


\bibitem{F79} H. B. Foxby, {\em Bounded complexes of flat modules},
J. Pure Appl. Algebra, {\bf 15}, (1979), 149--172.

\bibitem{FI} H. B. Foxby and S. Iyengar, {\em Depth and amplitude for unbounded
complexes}, Commutative algebra (Grenoble/Lyon, 2001), 119--137,
Contemp. Math. {\bf 331}, Amer. Math. Soc. Providence, RI, 2003.

\bibitem{Fr} A. Frankild, {\em Vanishing of local homology modules},
Math. Z., {\bf 244}, (3), (2003), 615--630.

\bibitem{G} A. Grothendieck, {\em Local cohomology}, Lecture notes in Math. {\bf 41} Springer Verlag, 1967.

\bibitem{Ha} R. Hartshorne,  {\em Residues and Duality}, Lecture Notes in math.
{\bf20}, Springer-Verlag, Heidelberg, 1966.

\bibitem{H} H. Holm, {\em Gorenstein homological dimensions}, J. Pure Appl.
Algebra, {\bf189}, (2004), 167-193.







\bibitem{RG} M. Raynaud and L. Gruson, {\em Crit\`{e}res de platitude et de projectivit\'{e}.
Techniques de ``platification'' d'un module}, Invent. Math. {\bf
13}, (1971), 1--89.

\bibitem{SSH} P. Sahandi and T. Sharif, {\em Dual of the Auslander-Bridger formula and GF-perfectness}, Math. Scand., {\bf 101}, (2007),
5--18.

\bibitem{SS} A. Singh and I. Swanson, {\em Associated primes of local cohomology modules and of
Frobenius powers}, Int. Math. Res. Not. {\bf 33}, (2004),
1703--1733.


\bibitem{Y} S. Yassemi, {\em Width of complexes of modules}, Acta Math. Vietnam.
{\bf 23}(1), (1998) 161--169.









\end{thebibliography}
\end{document}